\documentclass[11pt]{article}
\usepackage{amssymb}
\usepackage{latexsym}
\usepackage{amsthm}
\usepackage{amscd}
\usepackage{amsmath}
\usepackage{diagrams}
\usepackage{indentfirst}

\theoremstyle{definition}
\newtheorem{thm}{Theorem}[section]
\newtheorem{lem}[thm]{Lemma}
\newtheorem{th-def}[thm]{Theorem-Definition}
\newtheorem{cor}[thm]{Corollary}
\newtheorem{defn-lem}[thm]{Definition-Lemma}

\newtheorem{ex}[thm]{Example}
\newtheorem{prop}[thm]{Proposition}

\newtheorem{rem}[thm]{Remark}

\newtheorem{defn}[thm]{Definition}
\numberwithin{equation}{section}

\allowdisplaybreaks[4]

\def \Q{{\mathbb Q}}

\def \Z{{\mathbb Z}}

\def \P{\mathbb P}

\def\map#1.#2.{#1 \longrightarrow #2}
\def\rmap#1.#2.{#1 \dasharrow #2}

\DeclareMathOperator{\Alb}{Alb}
\DeclareMathOperator{\Iso}{Iso}
\DeclareMathOperator{\Spec}{Spec}
\DeclareMathOperator{\Jac}{Jac}

\DeclareMathOperator{\gen}{gen}

\DeclareMathOperator{\Stab}{Stab}

\DeclareMathOperator{\Aut}{Aut}

\def\fb#1.{\underset #1 \to \times}
\def\pr#1.{\Bbb P^{#1}}
\def\ring#1.{\mathcal O_{#1}}
\def\mlist#1.#2.{{#1}_1,{#1}_2,\dots,{#1}_{#2}}

\def\uloopr#1{\ar@'{@+{[0,0]+(-4,5)} @+{[0,0]+(0,10)}
@+{[0,0]+(4,5)}}
  ^{#1}}
\def\dloopr#1{\ar@'{@+{[0,0]+(-4,-5)} @+{[0,0]+(0,-10)}
@+{[0,0]+(4,-5)}}
  _{#1}}

\def\rloopd#1{\ar@'{@+{[0,0]+(5,4)} @+{[0,0]+(10,0)}
@+{[0,0]+(5,-4)}}
  ^{#1}}
\def\lloopd#1{\ar@'{@+{[0,0]+(-5,4)} @+{[0,0]+(-10,0)}
@+{[0,0]+(-5,-4)}}
  _{#1}}

\long\def\ignore#1{}
\long\def\ignore#1{#1}
\title{The ordinarity of an isotrivial elliptic fibration}
\author{}
\date{}

\begin{document}
\maketitle
\begin{center}
{\Large Junmyeong Jang} \footnote{\noindent J.Jang : School of
Mathematics, Korea Institute for Advanced Study, Hoegiro 87,
Dongdaemun-gu, Seoul
130-722, Korea jjm@kias.re.kr}
\end{center}

\medskip
     \noindent
     {\bf Abstract} :
     In this paper, we study the ordinarity of an isotrivial
     elliptic surface defined over a field of positive
     characteristic. If an isotrivial elliptic fibration $\pi:X
     \to C$ is given, $X$ is ordinary when the common fiber of
     $\pi$ is ordinary and a certain finite cover of the base $C$
     is ordinary. By this result, we may obtain the ordinary
     reduction theorem for some kinds of isotrivial elliptic
     fibrations.\\
     \rule{14.5cm}{0.3mm}
     \section{Introduction}
     Let $k$ be an algebraically closed field of characteristic
     $p>0$ and $A$ be an abelian variety of $d$-dimensional over
     $k$. The number of $p$-torsion points of $A(k)$ is given by
     $p^{\gamma}(0\leq \gamma \leq d)$ and $\gamma$ is called the
     $p$-rank of $A$. By definition, $A$ is ordinary if $\gamma =d$.
     There are several equivalent conditions for $A$ to be ordinary.
     Let $F_{A}$ be the absolute Frobenius morphism of $A$. $A$ is
     ordinary if and only if $F_{X} ^{*} : H^{1}(\mathcal{O}_{A})
     \to H^{1} (\mathcal{O}_{A})$ is bijective. Also, $A$ is
     ordinary if and only if the Frobenius morphism on
     $H^{1}(W\mathcal{O}_{A})$ is bijective, here
     $W\mathcal{O}_{A}$ is the sheaf of the Witt vectors of
     $\mathcal{O}_{A}$. (\cite{I},p.651) The ordinarity of $A$ is
     equivalent to that the Newton polygon of $A$ and the Hodge
     polygon of $A$ coincide at every level. The definition of
     ordinarity can be generalized to an arbitrary proper smooth
     variety over $k$. Let us recall the definition of an ordinary variety. (\cite{IR})
     Let $X$ be a proper smooth variety of $d$-dimensional over $k$ and $F_{X}$ be
     the absolute Frobenius morphism on $X$. Let
     $$ 0 \to \mathcal{O}_{X} \overset{d}\to \Omega _{X/k}^{1} \to \cdots \overset{d}\to
     \Omega _{X/k} ^{d} \to 0$$
     be the deRham complex of $X/k$. Consider the push forward of
     the deRham complex by $F_{X}$
     $$ 0 \to F_{X*}\mathcal{O}_{X} \overset{d}\to F_{X*}\Omega _{X/k}^{1} \to \cdots \overset{d}\to
     F_{X*}\Omega _{X/k} ^{d} \to 0.$$
     In this complex, the exterior differential $d$ is
     $\mathcal{O}_{X}$-linear, so the kernel and the image of $d$ are
     coherent $\mathcal{O}_{X}$-modules. In fact, they are vector bundles. We denote the kernel and
     image of $d : F_{X*}\Omega_{X/k}^{i-1} \to
     F_{X*}\Omega_{X/k}^{i}$ by $Z\Omega^{i-1}_{X/k}$ and $B\Omega
     ^{i} _{X/k}$ respectively.
     \begin{defn}
     $X$ is (Bloch-Kato-Illusie-Raynaud) ordinary if $H^{i}(B\Omega _{X/k}^{j})=0$ for all $i$
     and $j$.
     \end{defn}
     When the crystalline cohomology of $X/k$,
     $H^{i}_{cris}(X/W(k))$ is torsion free for each $i$, $X$ is
     ordinary if and only if the Hodge polygon and the Newton
     polygon coincide at every level. (\cite{IR} p.209)  Since $Z\Omega _{X/k}^{0}
     \subset F_{X*}\mathcal{O}_{X}$ is equal to the image of
     $F_{X}^{*} : \mathcal{O}_{X} \hookrightarrow
     F_{X*}\mathcal{O}_{X}$, there is an exact sequence
     $$0 \longrightarrow \mathcal{O}_{X} \overset{F_{X}^{*}}{\longrightarrow}
     F_{X*}\mathcal{O}_{X} \longrightarrow B\Omega ^{1} \longrightarrow 0.$$
     If $X$ is ordinary, $H^{i}(B\Omega^{1})=0$ for each $i$, so the Frobenius morphism on
     $H^{i}(\mathcal{O}_{X})$ is bijective for each $i$.
     When $X$ is an abelian variety or a curve, the converse also holds.

     Let $F$ be a number field and $E$ be an elliptic curve
     defined over $F$. The following theorem is well known.
     \begin{thm}[Ordinary reduction theorem, \cite{Se}]
     The set of ordinary reduction places of $F$ for $E$ has a
     positive density. Moreover for a suitable finite field
     extension $F'/F$, the set of ordinary reduction places of
     $F'$ for $E$ has density 1.
     \end{thm}
      The ordinary reduction theorem holds for all
     abelian surfaces (\cite{Og}) and for some kinds of abelian
     varieties of higher dimension. (\cite{Pi} , \cite{No}) It is
     conjectured that the ordinary reduction theorem holds for all
     abelian varieties of all dimensions defined over a number field. We may
     think of the ordinary reduction problem for an arbitrary
     proper smooth variety defined over a number field. It is known that for an
     K3-surface, the ordinary reduction theorem holds. (\cite{BZ})

     In this paper we study the ordinarity of a proper smooth
     surfaces which admit an isotrivial elliptic fibration to a proper
     smooth curve. An isotrivial elliptic fibration $\pi : X \to C$ is
     given by a desingularization of a quotient of a product of
     an elliptic curve and another curve $E \times D$ by a finite group $G$ which acts on both curves
     faithfully. (Section 3) Hence the finite group $G$ should be a semi-product
     of a finite abelian group of rank at most two and a cyclic
     rotation group which is $\Z/6,\Z/4,\Z/3,\Z/2$ or the trivial group.
     If the order of $G$ is relatively prime to the characteristic
     of the base field, the ordinarity of the surface $X$ can be
     determined by the ordinarity of the common fiber $E$ and a
     certain finite cover of $C$. (Theorem 4.1) Using this result
     we can prove the ordinary reduction theorem holds
     for some kinds of isotrivial elliptic surfaces. (Corollary 4.10)
     And also we construct an ordinary surface which admits an isotrivial
     fibration such that the common fiber is not ordinary. (Example 4.3, 4.4)\\[0.6cm]
     {\bf Acknowledgment}\\
          The author appreciates to Professor J.Keum, Professor T.Katsura and D.Hwang for
     helpful discussion.

     \section{Ordinary surfaces}
     Let $k$ be an algebraically closed field of characteristic
     $p>0$. $X$ is a smooth variety of $d$-dimensional over $k$.
     Let $x$ be a closed point in $X$. \'{E}tale locally around $x$, we may assume $X$
     is isomorphic to $\Spec k[t_{1}, \cdots , t_{d}]$. The $i$-th
     cohomology sheaf of the deRham complex of $X$, $\mathcal{H}^{i}\Omega _{X/k}=Z\Omega^{i}/B\Omega
     ^{i}$ is a vector bundle and the set of elements $t_{j_{1}}^{p-1} \cdots
     t_{j_{i}}^{p-1}dt_{j_{1}}\cdots dt_{j_{i}}$ ($1\leq j_{1} <
     j_{2} < \cdots < j_{i} \leq d$) forms a local basis of $\mathcal{H} ^{i}\Omega _{X/k}$.
     The Cartier morphism $C :
     \mathcal{H} ^{i}\Omega _{X/k} \to \Omega ^{i}_{X/k}$ defined
     by $$t_{j_{1}}^{p-1} \cdots t_{j_{i}}^{p-1} dt_{j_{1}} \cdots
     dt_{j_{i}} \mapsto dt_{j_{1}} \cdots dt_{j_{i}}$$
     is globally well defined and is an $\mathcal{O}_{X}$-linear isomorphism. (\cite{I}, Chapter 0)

     Assume $X$ is a smooth surface over $k$. Let us think of a pairing
     $$\varphi ': F_{X*}\mathcal{O}_{X} \otimes B\Omega^{2} \to
     \Omega ^{2} _{X/k}$$
     defined by $\varphi '(\alpha \otimes \omega) = C(\alpha
     \omega)$.
     $\mathcal{O}_{X} \subset F_{X*}\mathcal{O}_{X}$ is the
     kernel of $\varphi '$, so there is the induced pairing $\varphi : B\Omega
     ^{1} \otimes B\Omega ^{2} \to \Omega ^{2} _{X/k}$ which is
     perfect. By the Serre duality, we have
     $H^{i}(B\Omega ^{1}) = H^{2-i}(B\Omega ^{2})$.
     \begin{lem}
     Let $X$ be a proper smooth surface over $k$. $X$ is ordinary
     if and only if $F_{X}^{*} : H^{i}(\mathcal{O}_{X}) \to
     H^{i}(\mathcal{O}_{X})$ is bijective for $i=1,2$.
     \end{lem}
     Assume $X$ is a proper smooth variety and $Y \subset X$ is a
     smooth subvariety of codimension $\geq 2$. Let $\tilde{X} \to
     X$ be the blow-up along $Y$. Then $\tilde{X}$ is ordinary if
     and only if both of $X$ and $Y$ are ordinary. (\cite{JR},
     p.118) Hence for surfaces, the ordinarity is a birational
     invariant.
     \begin{lem}
     Let $f: X' \to X$ be a birational morphism of proper smooth surfaces. $X$ is ordinary
     if and only if $X'$ is ordinary.
     \end{lem}
     By lemma 2.1., if $H^{1}(\mathcal{O}_{X}) =
     H^{2}(\mathcal{O}_{X})=0$, $X$ is ordinary. In particular, a rational surface or a
     classical Godeaux surface is ordinary. Moreover for a rational surface or a Godeaux surface
     defined over a number field, the ordinary
     reduction theorem is trivial. If $H^{2}(\mathcal{O}_{X})=0$,
     the image of the Albanese map $\phi : X \to \Alb X$ is a smooth
     curve $C$. (\cite{Be},p.86) $\Alb X = \Jac C$ and $\phi ^{*} :
     H^{1}(\mathcal{O}_{C}) \to H^{1}(\mathcal{O}_{X})$ is
     isomorphic. Therefore $X$ is ordinary if and only if $C$ is
     ordinary.
     \begin{prop}
     Let $F$ be a number field and $X$ be a proper smooth surface
     defined over $F$.\\
     (a) If the Kodaira dimension of $X$, $\kappa (X) = -\infty$,
     i.e. $X$ is birationally ruled and $\pi : X \to C$ is the
     structure morphism, then the ordinary reduction theorem for
     $X$ is equivalent to the ordinary reduction theorem for
     $C$.\\
     (b) If $\kappa (X) =0$, the ordinary reduction theorem holds for $X$.
     \end{prop}
     \begin{proof}
     If $\kappa (X)=-\infty$, $\Alb X =\Jac C$ and $X \to C \hookrightarrow \Jac
     C$ is the Albanese map of $X$. Hence (a) follows by the above argument. Assume $\kappa (X)
     =0$. Because of Lemma 2.2, we may assume $X$ is minimal.
     Assume $X$ is an abelian surface, a K3 surface, an Enrique surface or
     a bielliptic surface. If $X$ is an abelian surface or a K3 surface,
     the ordinary reduction theorem holds. (Section 1.) If
     $X$ is an Enrique surface, there exists an \'{e}tale double cover
     $X' \to X$ where $X'$ is a K3 surface. If $\upsilon$ is a
     place of $F$ whose residue characteristic$\neq 2$ and the
     reduction $X' _{\upsilon}$ is ordinary, $X_{\upsilon}$
     is also ordinary. Hence the ordinary reduction theorem holds for
     $X$. If $X$ is a bielliptic surface, $X$ admits a smooth
     elliptic fibration $\pi : X \to C$ to an elliptic curve, $C$.
     Here all the fibers of $\pi$ are isomorphic to each
     other and there exists a finite Galois \'{e}tale cover $D \to
     C$ such that the fiber product $X \times _{C} D$ is
     isomorphic to a trivial fibration $E \times _{k} D$, here
     $E$ is the common fiber. $E \times _{k} D \to X$ is
     finite Galois \'{e}tale and $E \times _{k} D$ is an abelian
     surface, so the ordinary reduction theorem holds for $E_{k}D$.
     Hence by a similar argument to the above, the ordinary
     reduction theorem holds for $X$.
     \end{proof}
     Now we consider the ordinarity of fibred surfaces.
     The fiber product of two curves, $X= C \times _{k} D$ is
     ordinary if and only if both of $C$ and $D$ are ordinary.
     (\cite{Ek},p.91) Indeed, in the case of a product of two curves,
     $$H^{1}(\mathcal{O}_{C \times D}) = H^{1}(\mathcal{O}_{C})
     \oplus H^{1}(\mathcal{O}_{D}) \textrm{ and } H^{2}(\mathcal{O}_{C
     \times D}) = H^{1}(\mathcal{O}_{C}) \otimes
     H^{1}(\mathcal{O}_{D}).$$ Since the Frobenius morphism acts
     separately, the claim follows.
     If we regard the fiber product $X=C \times _{k} D$ as a fibred
     surface equipped with a trivial fibration $X \to D$, then
     $X$ is ordinary if and only if both of the base and the generic
     fiber are ordinary. However this is not true in general.
     There is a non-ordinary 3-fold which admits a fibration over $\P
     ^{1}$ of which every smooth fiber is ordinary. (\cite{JR} p.119)
     \begin{ex}
     There is also a
     non-ordinary fibred surface over $\P ^{1}$ of which each smooth fiber
     is ordinary. Assume the characteristic of $k>2$.
     Assume $C$ is an ordinary elliptic curve over $k$
     and $D$ is a supersingular elliptic curve over $k$. $A=C \times _{k} D$ is
     an abelian surface of p-rank 1. The height of the formal
     Brauer group of $A$ is 2. Let $X$ be the Kummer surface
     of $A$. The height of $X$ is equal to the height of $A$,
     (\cite{GK},p.109) so $X$ is not ordinary. But $X$ admits a Kummer
     fibration $\pi : X \to \P ^{1}$ given by the quotient of the
     projection $A \to D$ by the involution. Each smooth fiber of
     $\pi$ is isomorphic to $C$, so ordinary, and there are 4 singular fibers
     which are of type $I _{0} ^{*}$ in Kodaira's classification.
     \end{ex}
     In the above example the image of singular fibers in $\P ^{1}$ consists of 4
     points. The double cover of $\P ^{1}$ ramified over those 4
     points is isomorphic to $D$. Hence although the base itself
     is ordinary, the fibration encodes information about $D$
     which is not ordinary. Later we will see this
     phenomenon in detail.

     Let $\pi : X \to C$ be a fibered surface. Since $\pi _{*}
     \mathcal{O}_{X} = \mathcal{O}_{C}$, considering the Leray spectral
     sequence, $\pi ^{*} : H^{1}(\mathcal{O}_{C}) \to H^{1}(\mathcal{O}_{X})$ is injective. And the diagram
     $$\begin{diagram}
     X & \rTo ^{F_{X}} & X\\
     \dTo <{\pi} & & \dTo <{\pi}\\
     C & \rTo ^{F _{C}}& C
     \end{diagram}$$
     is commutative, so if $$F_{X} ^{*} : H^{1}(\mathcal{O}_{X}) \to H^{1}(\mathcal{O}_{X})$$ is bijective,
     $$F_{C}^{*} : H^{1}(\mathcal{O}_{C}) \to H^{1}(\mathcal{O}_{C})$$ is bijective. Hence if $X$ is ordinary, $C$ is also ordinary.
     Assume $\pi$ is generically smooth.
     \begin{defn}
     We call $\pi$ is generically ordinary if almost all fibers are ordinary.
     \end{defn}
     $\pi$ is generically ordinary if and only if the generic fiber of $\pi$ is ordinary by the Grothendieck specialization theorem.
     \begin{prop}
     Let $\pi : X \to C$ be an elliptic fibration. If $H^{2}(\mathcal{O}_{X}) \neq 0$ and $X$ is ordinary,
     then $\pi$ is generically ordinary.
     \end{prop}
     \begin{proof}
     Let us consider the Frobenius diagram of $\pi$
     \begin{equation}\label{Frodia} \begin{diagram}
     X & \rTo ^{F_{X/C}} & X^{(p)} & \rTo ^{W} & X\\
     & \rdTo <{\pi} & \dTo <{\pi ^{(p)}}  & & \dTo <{\pi}\\
     & & C & \rTo ^{F _{C}} & C.
     \end{diagram}\end{equation}
     Here the right square is cartesian and $F_{X}=W \circ F_{X/C}$.
     There is an exact sequence of $\mathcal{O}_{X ^{(p)}}$-modules
     $$0 \to \mathcal{O}_{X^{(p)}} \overset{F_{X/C} ^{*}}{\to} F_{X/C*} \mathcal{O}_{X} \to B \to 0.$$
     Note that the restriction of the cokernel $B$ to a smooth fiber $X_{x}$ over $x \in C$ is just $B\Omega _{X_{x}/k} ^{1}$.
     $B$ is not $\mathcal{O}_{C}$-flat at a non-reduced fiber in general.
     Let $N = R^{1}\pi _{*} \mathcal{O}_{X}$. $N$ is a line bundle on $C$.
     We have the long exact sequence associated to the above exact sequence via $\pi _{*}$,
     \begin{equation}\label{exac}
     0 \to \mathcal{O}_{C} \simeq \mathcal{O}_{C} \to \pi^{(p)}_{*} B \to F_{C}^{*} N
     \overset {F_{X/C} ^{*}}{\to} N \to R^{1} \pi ^{(p)}_{*} B
     \to 0.\end{equation}
     Note that $H^{2}(\mathcal{O}_{X}) = H^{1}(N)$ and $H^{2}(\mathcal{O}_{X ^{(p)}}) = H^{1}(F_{C}^{*} N)$.
     If $\pi$ is generically ordinary, in (\ref{exac}), $F_{X/C}^{*}$ is injective, otherwise $F_{X/C} ^{*} =0$.
     Since
     $$F_{X} ^{*} = F_{X/C} ^{*} \circ W ^{*} : H^{2}(\mathcal{O}_{X}) \to H^{2}(\mathcal{O}_{X^{(p)}}) \to H^{2}(\mathcal{O}_{X}),$$
     if $\pi$ is not generically ordinary, $F_{X} ^{*}| H^{2}(\mathcal{O}_{X}) =0$, so $X$ is not ordinary.
     \end{proof}
     \begin{rem}
     Because the moduli space of elliptic curves is 1-dimensional and the generic elliptic curve is ordinary,
     if an elliptic fibration is not generically ordinary, it should be isotrivial. (Section 3.)
     If $H^{2}(\mathcal{O}_{X}) =0$ or the fiber
     genus is large, the above proposition may fail. We will see counterexamples later.
     \end{rem}

     \section{Isotrivial elliptic fibrations}
     Let $k$ be an algebraically closed field and char$(k)\neq 2,3$. $\pi : X \to C$ is an elliptic fibration
     over $k$.
     \begin{defn}
     We say $\pi$ is isotrivial if all
     the smooth fibers of $\pi$ are isomorphic to each other.
     \end{defn}
     We recall a systematical construction of
     isotrivial elliptic fibrations. (\cite{Be} chapter 6, \cite{Ser})
     Let $\pi : X \to C$ be an isotrivial elliptic fibration and $E$ be the common fiber
     of $\pi$.
     Let $U=C-$\{the image of the singular fibers of $\pi$\}.
     There exists a finite cover $U'$ of $U$ such that
     $$U' \times _{U} V = U' \times _{k} E.$$
     Let $C'$ be the proper smooth model of $U'$. We may assume $f : C' \to C$ is Galois. Let $H$ be the
     Galois group of $C'/C$ and $X' = X \times _{C} C'$. $X'$ has an $H$-action coming from the $H$-action of $C'$. This $H$-action
     on $X'$ is compatible with the $H$-action on $C'$ via the projection $X' \to C'$. Since $X'$ is birationally equivalent to
     $C' \times _{k} E$ and $C' \times _{k} E$ is minimal, there is an $H$-action on $C' \times _{k} E$
     $$\rho : H \times C' \times _{k} E \to C' \times _{k} E$$
      which is compatible with
     the action on $C'$ via the projection. By \cite{Be},p.102 with a slight modification,
     there exist a finite \'{e}tale cover $g : D \to C'$ and
     a finite group $G$ which acts on both of $D$
     and $E$ such that
     $$D/G \simeq C'/H = C \textrm{ and }(D \times _{k} E) /G \simeq
     (C' \times _{k}E)/H.$$
     Here the action of $G$ on $D \times _{k} E$ is the product of the actions on $D$ and $E$.
     We may assume the actions of $G$ on $E$ and
     $D$ are faithful. $(D \times _{k} E)/G$ may have isolated singularities
     and is birationally equivalent to $X$. Let $\Iso E$ be the
     group of $k$-variety automorphisms of $E$ and $\Aut E$ be the
     group of elliptic curve automorphisms of $E$. $\Aut E = \Z/2$
     if $j(E) \neq 0,1728$, $\Aut E = \Z/4$ if $j(E) = 1728$ and
     $\Aut E = \Z/6$ if $j(E)=0$. (\cite{Si},p.103) $\Iso E = E(k)
     \rtimes \Aut E$. Since $G$ acts faithfully on $E$, $G \subset
     \Iso E$ and $G = T \rtimes R$ where $T \subset E(k)$ and $R$ is embedded into
     $\Aut E$ by the projection $\Iso E \to \Aut E$, so $R$ is cyclic.
     Adjusting the $0$-point of $E$, we may regard $R \subset \Aut E$.
     In particular $G$ is solvable.
     Let $Y= (E \times D)/G$. We assume the characteristic of $k$ does not divides $|G|$.
     $T = \Z/ n_{1} \oplus \Z/
     n_{2}$ where $n_{2} | n_{1}$ and $R= \Z /6, \ \Z/4, \ \Z/3, \
     \Z/2,$ or trivial. Let
     $f: \tilde{X} \to Y$ be the standard desingularization which resloves the quotient singularities
     by Hirzebruch-Jung stirngs. (\cite{Ser},p.64)
     $\tilde{\pi} : X \to C = D/G$ is an isotrivial elliptic
     fibration. In general $\tilde{\pi}$ is not relatively minimal.
     Since $|G|$ is relatively prime to the characteristic
     of the base field, each singularity of $Y$ is a
     rational singularity, so $f^{*} : H^{i}(\mathcal{O}_{Y}) \to
     H^{i}(\mathcal{O}_{\tilde{X}})$ is isomorphic for each $i$.
     Let us think of the $G$-action on $E$. The following lemma is well-known.
     \begin{lem}
     A finite abelian group $G$ acts on a variety $X$ over $k$. Assume $|G|$
     is relative prime to the characteristic of the base field $k$.
     Let $Y$ be the quotient $X/G$ and $p : X \to Y$ be
     the projection. Then $p _{*} \mathcal{O}_{X} = \oplus _{\chi}
     L_{\chi}$ where $L_{\chi}$ is a coherent
     $\mathcal{O}_{Y}$-module on which $G$ acts through the character
     $\chi : G \to k^{*}$
     \end{lem}
     Let a finite group $G$ act on a variety $S$ and $T = S/G$. If
     $G$ is solvable and $|G|$ is relative prime to the
     characteristic of the base field, by the above lemma, $H^{i}(\mathcal{O}_{T})
     = H^{i}(\mathcal{O}_{S}) ^{G}$.

     Let $E' = E/T$. Since $T$
     is a translation group of $E$, $E'$ is an elliptic curve and
     $E \to E'$ is isogeny. If $R$ is $\Z/3$, $\Z/4$ or $\Z/6$,
     $E'$ is isomorphic to $E$. Assume $R$ is not trivial, and $g : E'
     \to E'/R = \P ^{1}$ is the canonical map. The following lemma is trivial.
     \begin{lem}
     (a) If $R=\Z /2$, $g_{*} \mathcal{O}_{E'} = \mathcal{O}_{\P
     ^{1}} \oplus \mathcal{O}_{\P ^{1}} (-2)$ and $1 \in \Z/2$
     acts on $\mathcal{O}_{\P ^{1}}(-2)$ by multiplying $-1$. There are 4 ramification points
     of ramification index 2.\\[0.15cm]
     (b) If $R=\Z /3 $, $g_{*} \mathcal{O}_{E'} = \mathcal{O} _{\P
     ^{1}} \oplus \mathcal{O}_{\P ^{1}}(-1) \oplus \mathcal{O}
     _{\P ^{1}}(-2)$ and there is a primitive 3rd root of unity
     $\xi _{3}$ such that $1 \in \Z/3$ acts on $\mathcal{O}_{\P ^{1}}
     (-2)$ by multiplying $\xi _{3}$. There are 3 ramification points of ramification index
     3.\\[0.15cm]
     (c) If $R=\Z /4 $, $g_{*} \mathcal{O}_{E'} = \mathcal{O} _{\P
     ^{1}} \oplus \mathcal{O}_{\P ^{1}}(-1) ^{\oplus 2} \oplus \mathcal{O}
     _{\P ^{1}}(-2)$ and there is a primitive 4th root of unity
     $\xi _{4}$ such that $1 \in \Z/4$ acts on $\mathcal{O}_{\P ^{1}}
     (-2)$ by multiplying $\xi _{4}$. There are 2 ramification points of ramification index 4
     and 2 ramification points of ramification index 2 which are conjugate to each
     other.\\[0.15cm]
     (d) If $R=\Z /6 $, $g_{*} \mathcal{O}_{E'} = \mathcal{O} _{\P
     ^{1}} \oplus \mathcal{O}_{\P ^{1}}(-1) ^{\oplus 4}
     \oplus  \mathcal{O}
     _{\P ^{1}}(-2)$ and there is a primitive 6th root of unity
     $\xi _{6}$ such that $1 \in \Z/6$ acts on $\mathcal{O}_{\P ^{1}}
     (-2)$ by multiplying $\xi _{6}$. Here $\xi _{6}^{2} = \xi _{3}$.
     There are 1 point of ramification index 6,
     2 points of ramification index 3 and 3 points of ramification index 2.
     \end{lem}
     Let $\pi : X \to C$ be the minimal relative model of $\tilde{\pi} : \tilde{X} \to C$.
     $\tilde{X} \to X$ contracts every $-1$-curve in a fiber of $\tilde{\pi}$.
     Let us recall the figure of a singular fiber of $\pi : X \to C$.
     (\cite{BPV},p.157, \cite{Ogu}, \cite{Ser}) Since $|G|$ is relative prime to the characteristic of $k$, for any $x
     \in D$, the stabilizer of $x$, $\Stab _{x}(G) \subset G$ is a cyclic group.
     By an easy calculation, we can check $\Stab _{x}(G) \subset
     T$ or $\Stab _{x}(G) \subset R$. Let $y \in C$ be the image
     of $x$. $Y_{y}$ is the fiber of $ Y
     \to C$ over $y$, $\tilde{X}_{y}$ is the fiber of $\tilde{\pi}
     : \tilde{X} \to C$ over $y$ and $X_{y}$ is the fiber of $\pi : X \to C$ over $y$.
     $E_{l}(X_{y})$ is the $l$-adic Euler characteristic of
     $X_{y}$,
     $$E_{l}(X_{y}) = \sum (-1) ^{i} \dim _{\Q _{l}} H ^{i} (X_{y},
     \Q _{l}).$$
     Here $l$ is a prime number different from the characteristic of $k$.
     We have several cases.\\

     (i) Assume $\Stab _{x}(G) \subset
     T$. $Y_{y}$ is $|\Stab _{x}(G)| (E/
     \Stab _{x}(G))$. Note that every point $Y_{y}$
     is non-singular on $Y$. Hence
     $X_{y}=\tilde{X}_{y}=Y_{y}$ and $E_{l}(X_{y})=0$,\\

     (ii) Assume $\Stab _{x}(G) \subset R$ and isomorphic to $\Z /2$.
     $Y_{y}$ has 4 singular points of type $A_{1}$. After the
     Hirzebruch-Jung resolution, $\tilde{X}$ is a fiber of type $I_{0}^{*}$
     which is minimal. $X_{y}=\tilde{X}_{y}$ and
     $E_{l}(X_{y})$=6.\\

     (iii) Assume $\Stab _{x}(G) \subset R$ and isomorphic to $\Z /3$.
     Let
     $$\hat{\mathcal{O}}_{D,x} = \underset{\leftarrow}{\lim} \mathcal{O}
     _{D,x}/m_{x}^{n} \simeq k[[t]].$$ Assume $1 \in \Z/3$ acts on $\hat{\mathcal{O}}
     _{D,x}$ by
     multiplying $\xi _{3}$. Then $Y_{y}$ has 3 singular points
     of type $A_{3,1}$ and $\tilde{X}_{y}$ is $3L_{1} + L_{2} + L_{3} +
     L_{4}$, where each $L_{i}$ is a non-singular rational curve
     and the intersection matrix is
     $$\left (
     \begin{array}{cccc}
     -1 & 1 & 1 & 1\\
     1& -3 & 0& 0\\
     1&0&-3&0\\
     1&0&0&-3
     \end{array}
     \right ).$$
     After blowing down $L_{1}$, $X_{y}$ is type of $IV$ and
     $E_{l}(X_{y})=4$. Assume $1 \in \Z/3$ acts on $\hat{\mathcal{O}}_{D,x}$ by
     multiplying $\xi _{3} ^{-1}$. $Y_{y}$ has 3 singular points
     of type $A_{2}$. $\tilde{X}_{y}=X_{y}$ is a fiber of type $IV
     ^{*}$. $E_{l}(X_{y})=8$.\\

     (iv) Assume $\Stab _{x}(G) \subset R$ and isomorphic to $\Z /4$.
     Assume $1 \in \Z/4$ acts on $\hat{\mathcal{O}}_{D,x}$ by multiplying $\xi _{4}$.
     $Y_{y}$ has 2 singular points of type $A_{4,1}$ and 1 singular point of type $A_{1}$.
     $\tilde{X}_{y}$ is
     $4L_{1} + L_{2} + 2L_{3} + L_{4}$, where each $L_{i}$ is
     a non-singular rational curve and the intersection matrix is
     $$\left (
     \begin{array}{cccc}
     -1 & 1 & 1 & 1\\
     1& -4 & 0& 0\\
     1&0&-2&0\\
     1&0&0&-4
     \end{array}
     \right ).$$
     After blowing down twice, the fiber $X_{y}$ is of type $III$.
     $E_{l}(X_{y})=3$. Assume $1 \in \Z/4$ acts on $\hat{\mathcal{O}}_{D,x}$ by
     multiplying $\xi _{4} ^{-1}$. $Y_{y}$ has 1 singular point of
     type $A_{1}$ and 2 singular points of type $A_{3}$. $\tilde{X}_{y}$ is minimal
     and of type $III ^{*}$. $X_{y}=\tilde{X}_{y}$ and
     $E_{l}(X_{y})=9$.\\

     (v) Assume $\Stab_{x}(G) \subset R$ and isomorphic to $\Z /6$.
     Assume $1 \in \Z/6$ acts on $\hat{\mathcal{O}}_{D,x}$ by multiplying $\xi
     _{6}$. $Y_{y}$ has 1 singular point of type $A_{6,1}$, 1 singular point of type $A_{3,1}$
     and 1 singular point of type $A_{1}$. $\tilde{X} _{y}$ is
     $6L_{1} + L_{2} + 2L_{3} + 3L_{4}$, where each $L_{i}$ is
     a non-singular rational curve and the intersection matrix is
     $$\left (
     \begin{array}{cccc}
     -1 & 1 & 1 & 1\\
     1& -6 & 0& 0\\
     1&0&-3&0\\
     1&0&0&-2
     \end{array}
     \right ).$$
     After blowing down three times, $X_{y}$ is of type $II$ and
     $E_{l}(X_{y})=2$. Assume $1 \in \Z/6$ acts on
     $\hat{\mathcal{O}}_{D,x}$ by multiplying $\xi _{6} ^{-1}$. $Y_{y}$ has
     1 singular point of type $A_{5}$, 1 singular point of type $A_{2}$ and 1 singular point
     of type $A_{1}$. $\tilde{X} _{y}$ is minimal and of type $II ^{*}$.
     $X_{y}= \tilde{X}_{y}$ and $E_{l}(X_{y}) = 10$.\\

     Let $E_{l}(X) = \sum (-1)^{i} \dim _{\Q _{l}} H^{i} _{et} (X
     , \Q_{l})$. Then
     $$E_{l}(X) = \sum E_{l}(X_{y}),$$
     here $E_{l}(X_{y})=0$ if $X_{y}$ is smooth, so the right hand
     side is a finite sum. Since $X$ is relatively
     minimal, $K_{X} ^{2} =0$. By the Noether formula, we
     obtain
     $$12\chi (\mathcal{O}_{X}) = E_{l}(X).$$
     On the other hand, $H^{i}(\mathcal{O}_{X}) =
     H^{i}(\mathcal{O} _{Y})$
     and
     $$H^{1}(\mathcal{O}_{Y}) =
     H^{1}(\mathcal{O}_{C}) \oplus H^{1}(\mathcal{O}_{E})^{G} \textrm{ and }
     H^{2}(\mathcal{O}_{Y}) = (H^{1}(\mathcal{O}_{E}) \otimes
     H^{1}(\mathcal{O}_{D}))^{G}.$$
     Let $D' = D/T$. $R$ is the Galois group of $D'/C$.
     Let $\rho _{D'}$ be the $R$-action on $D'$.
     $g: D' \to C$ is the canonical quotient map. Let $n=
     |R|$ and $\chi : R \to k^{*}$ be a character satisfying $\chi (1) = \xi
     _{n}$. Then $g_{*}\mathcal{O}_{D'} = \oplus _{i=0}^{n-1}
     L_{i}$, where $R$ acts on $L_{i}$ through $\chi ^{i}$. Since $R$
     acts on $D'$ faithfully, each $L_{i}$ is a line bundle on $C$
     and $L_{0} =\mathcal{O}_{C}$. Note that if $n>2$, there is another
     $R$-action $\rho '_{D'}$ on $D'$ such that $\rho '_{D'} (1)$
     acts on $L_{1}$ by multiplying $\xi _{n} ^{-1}$. Let $a_{2}$
     be the number of branch points of ramification index 2 for $D'/C$. For
     $i=3,4,6$, let $a_{i}^{+}$ be the number of branch points of
     ramification index $i$ for $D'/C$ such that at a corresponding
     ramification point, $1 \in \Z/i \subset R$ acts by multiplying $\chi$.
     $a_{i}^{-}$ is the number of branch points such that at a corresponding
     ramification point, $1$ acts by multiplying $\chi ^{-1}$.
     \begin{prop}
     We assume all the above notations.\\
     (a) Assume $R$ is trivial. $Y$ is smooth $X=Y$.
     $H^{1}(\mathcal{O}_{X }) = H^{1}(\mathcal{O}_{E}) \oplus
     H^{1}(\mathcal{O}_{C})$, $H^{2}(\mathcal{O}_{X}) =
     H^{1}(\mathcal{O} _{E}) \otimes H^{1}(\mathcal{O}_{C})$.
     $R^{1}\pi _{*} \mathcal{O}_{X} = \mathcal{O}_{C}$. $\chi (\mathcal{O}_{X}) =
     E_{l}(X)=0$.\\[0.15cm]
     (b) Assume $R = \Z/2$.
     $H^{1}(\mathcal{O}_{X})= H^{1}(\mathcal{O}_{C})$ and
     $H^{2}(\mathcal{O}_{X}) = H^{1}(\mathcal{O}_{E}) \otimes
     H^{1} (L_{1})$. $R^{1} \pi _{*} \mathcal{O}_{X} = L_{1}$.
     $\deg L_{1} =- a_{2}/2$. $\chi (\mathcal{O}_{X})= a_{2}/2$ and
     $E_{l}(X) = 6a_{2}$. \\[0.15cm]
     (c) Assume $R= \Z/3$. $H^{1}(\mathcal{O}_{X}) =
     H^{1}(\mathcal{O}_{C})$ and $H^{2}(\mathcal{O}_{X}) =
     H^{1}(\mathcal{O}_{E}) \otimes H^{1}(L_{2})$. $R^{1} \pi _{*}
     \mathcal{O}_{X} = L_{2}$. $-\deg L_{2} =
     \chi(\mathcal{O}_{X}) = \frac{1}{3} (a_{3} ^{+} +2a_{3}
     ^{-})$ and $\deg L_{1} = -\frac{1}{3} (2a_{3}^{+} +
     a_{3}^{-})$. $E_{l}(X) = 4a_{3}^{+} + 8a_{3}^{-}$.\\[0.15cm]
     (d) Assume $R = \Z/4$. $H^{1}(\mathcal{O}_{X}) =
     H^{1}(\mathcal{O}_{C})$ and $H^{2}(\mathcal{O}_{X}) = H^{1}
     (\mathcal{O}_{E}) \otimes H^{1}(L_{3})$. $R^{1} \pi _{*}
     \mathcal{O}_{X} =L_{3}$. $-\deg L_{3} = \chi(\mathcal{O}
     _{X}) = \frac{1}{4}(a_{4} ^{+} + 3 a_{4} ^{-}+2a_{2})$ and $ \deg
     L_{1} = -\frac{1}{4}(3a_{4}^{+} + a_{4}^{-}+2a_{2})$. $\deg L_{2} =
     - \frac{1}{2}(a_{4} ^{+} + a_{4} ^{-})$. $E_{l}(X) = 3a_{4}
     ^{+} + 9 a_{4} ^{-}+6a_{2}$.\\[0.15cm]
     (e) Assume $R = \Z/6$. $H^{1}(\mathcal{O}_{X}) =
     H^{1}(\mathcal{O}_{C})$ and $H^{2}(\mathcal{O}_{X}) =
     H^{1}(\mathcal{O}_{E}) \otimes H^{1}(L_{5})$. $R^{1} \pi _{*}
     \mathcal{O}_{X} = L_{5}$. $-\deg L_{5} =
     \chi(\mathcal{O}_{X}) = \frac{1}{6}(a_{6}^{+} + 5
     a_{6}^{-1} + 2 a_{3}^{+} + 4 a_{3}^{-} + 3a_{2})$. $\deg
     L_{1} = - \frac{1}{6}(5a_{6}^{+} + a_{6}^{-} + 4a_{3}^{+}
     + 2a_{3}^{-} + 3a_{2})$. $\deg L_{3} = -\frac{1}{2}(a_{2} + a_{6}
     ^{+} + a_{6}^{-})$. $\deg L_{2} = -\frac{1}{3}(a_{6}^{+} +
     a_{3} ^{-} +2a_{6}^{-} +2a_{3}^{+})$. $\deg L_{4} =
     -\frac{1}{3} (2a_{6}^{+} + 2a_{3} ^{-} + a_{6}^{-} +
     a_{3}^{+})$. $E_{l}(X) = 2a_{6}^{+} + 10a_{6}^{-} + 4a_{3}^{+}
     +8a_{3}^{-} +6a_{2}$.
     \end{prop}
     \begin{proof}
     If $R$ is trivial, any $g \in G$ has no fixed point on $D \times E$, so $Y$
     is smooth and $X=Y$. Moreover $G=T$ and $T$ acts trivially on
     $H^{1}(\mathcal{O}_{E})$. Since
     $$\dim H^{1}(\mathcal{O}_{X})
     = \gen C + \dim H^{0}(R^{1} \pi _{*} \mathcal{O}_{X}) \textrm{ and }
     \dim H^{2}(\mathcal{O}_{X}) = \dim H^{1}(R^{1}\pi _{*}
     \mathcal{O}_{X}),$$ by the Riemann-Roch theorem, $\deg R^{1}
     \pi _{*}\mathcal{O}_{X} = 0$ and $\dim H^{0}(R^{1} \pi _{*}
     \mathcal{O}_{X})=1$. Hence $R^{1} \pi _{*} \mathcal{O}_{X}
     =\mathcal{O}_{C}$. This proves (a). Let $Z = (E \times D)/T$
     and consider the following diagram
     $$\begin{diagram}
     E \times D & \rTo ^{\alpha} & Z & \rTo ^{\beta} & Y\\
     \dTo >{p} & & \dTo >{q} & & \dTo >{r}\\
     D & \rTo ^{h}& D' & \rTo ^{g} & C.
     \end{diagram}$$
     Because $\alpha , \beta , h, g$ are finite morphisms,
     $$R^{1} r_{*} (\beta \circ \alpha)_{*} \mathcal{O}_{E \times
     D} = (g \circ h)_{*}  R^{1} p_{*} \mathcal{O}_{E \times D}.$$
     One hand, we have
     $$(R^{1} r_{*} (\beta
     \circ \alpha)_{*} \mathcal{O}_{E \times D})^{G} = R^{1} r_{*}
     \mathcal{O}_{Y}= R^{1} \pi_{*} \mathcal{O}_{X}.$$
     On the other
     hand,
     $$((g \circ h)_{*} R^{1} p_{*}\mathcal{O}_{E \times
     D})^{G} = (H^{1}(\mathcal{O}_{E}) \otimes
     \mathcal{O}_{D})^{G} = (H^{1}(\mathcal{O}_{E}) \otimes
     \mathcal{O}_{D'}) ^{R} = L_{n-1}.$$
     When $R =\Z/2$, $\pi$ has $a_{2}$ fibers of type $I_{0}^{*}$.
     Therefore
     $$E_{l}(X) = 6 a_{2} = 12 \chi(\mathcal{O}_{X}) =
     -12 \deg R^{1} \pi _{*} \mathcal{O}_{X}.$$
     This proves (b).
     Now let $\rho _{D}$ be the given action of $G$ on $D$.
     $\lambda : G \to G$ is an automorphicm of $G$ given by $(a,b)
     \in T \rtimes R \mapsto (a,-b)$. Let
     $$\rho '_{D} = \rho _{D}
     \circ \lambda : G \to \Aut D$$ and $\pi ' : X' \to C$
     be the minimal isotrivial elliptic fibration associated to the action $\rho
     _{E} \times \rho _{D}'$ on $E \times D$. Assume $R =\Z /3$.
     $\pi$ has $a_{3}^{+}$ fibers of type $IV$ and $a_{3} ^{-}$
     fibers of type $IV ^{*}$. Thus
     $$E_{l}(X) = 4a_{3}^{+} + 8
     a_{3} ^{-}.$$ On the other hand, $\pi '$ has $a_{3} ^{+}$
     fibers of type $IV ^{*}$ and $a_{3} ^{-}$ fibers of type
     $IV$, so
     $$E_{l}(X')=8a_{3}^{+}+4a_{3}^{-}\textrm{ and } R^{1} \pi '_{*} \mathcal{O}_{X'} = L_{1}.$$
     This proves
     (c). If $R= \Z/4$, let $g'' :D'' \to C$ be the intermediate double covering of
     $D/C$. Then $g''_{*} \mathcal{O}_{D''}= \mathcal{O}_{C} \oplus L_{2}$. Since
     the number of branch points of $g''$ is $a_{4}^{+} + a_{4}
     ^{-}$, $\deg L_{2} = -\frac{1}{2}(a_{4}^{+} + a_{4}^{-})$. The
     rest of proof of (d) is similar to that of (c). If $R =
     \Z/6$, let $g''' : D''' \to C$ be the intermediate triple
     cover of $D/C$. $g'''_{*} \mathcal{O}_{D'''} =
     \mathcal{O}_{C} \oplus L_{2} \oplus L_{4}$. $1 \in
     \Z/3 = Gal(D'''/C)$ acts on $L_{2}$ by multiplying $\xi _{6} ^{2} =
     \xi_{3}$. Let $x \in D'$ be a point of ramification index 3. Assume $2
     \in \Z /6=R$ is the generator of $\Z/3 \subset \Z/6$
      acting on $\mathcal{O}_{D',x}$ by multiplying $\xi _{3}$. $y = g(x)$ and $z$
     is the unique point in $D'''$ lying over $x$. Then $1=4 \in
     Gal(D'''/C)$ acts on $\mathcal{O}_{D''',z}$ by multiplying
     $\xi _{3} ^{2}$. Hence the number of ramification points of
     $D'''/C$ corresponding to the character $\chi _{3}$ is $a_{6}
     ^{+} + a_{3} ^{-}$. In the same way, the number of
     ramification points corresponding $\chi _{3} ^{2}$ is $a_{6}
     ^{-} + a_{3}^{+}$. By part (c),
     $$\deg L_{2} = -\frac{1}{3}(a_{6}^{+} +
     a_{3} ^{-} +2a_{6}^{-} +2a_{3}^{+})\textrm{ and }\deg L_{4} =
     -\frac{1}{3} (2a_{6}^{+} + 2a_{3} ^{-} + a_{6}^{-} +
     a_{3}^{+}).$$
     This finishes the proof.
     \end{proof}

     \section{Frobenius morphism of isotrivial elliptic
     fibrations}
     Let $k$ be an algebraically closed field of characteristic $p>3$.
     We resume the notations of the last section. Remind that $\pi ' :X' \to C$
     is the minimal isotrivial fibraiton associated to the
     $G$-action $\rho _{E} \times \rho '_{D}$ on $E \times D$.
     \begin{thm}
     (a) If $R$ is trivial, $X$ is ordinary if and only if both of
     $E$ and $C$ are ordinary.\\[0.15cm]
     (b) If $R = \Z /2$ and $X$ is not rational, $X$ is ordinary if and only if both of
     $E$ and $D'$ are ordinary.
     \\[0.15cm]
     (c) If $R = \Z /3$, unless both of $X$ and $X'$
     are rational, both of $X$ and $X'$ are ordinary if and
     only if both of $D'$ and $E$ are ordinary.\\[0.15cm]
     (d) If $R = \Z /4$, unless both of $X$ and $X'$ are rational,
     both of $X$, $X'$ are ordinary
     if and only if both of $E$ and $C$ are ordinary and $\gen D' - \gen
     D''$ = $p$-rank  of $D'$ $-$ $p$-rank of $D''$.\\[0.15cm]
     (e) If $R = \Z/6$, unless both of $X$ and $X'$ are rational, all of $X$, $X'$, $D''$ are
     ordinary if and only if both of $E$ and $C$ are ordinary
     and $\gen D' - \gen D'''$ = $p$-rank of $D'$ $-$ $p$-rank of
     $D'''$.
     \end{thm}
     \begin{proof}
     When $R$ is trivial, $$H^{1}(\mathcal{O}_{X}) =
     H^{1}(\mathcal{O}_{E}) \oplus H^{1}(\mathcal{O}_{C}) \textrm{
     and } H^{2}(\mathcal{O}_{X}) = H^{1} (\mathcal{O}_{E}) \otimes
     H^{1}(\mathcal{O}_{C}).$$ Since the Frobenius morphism acts
     separately, $F_{X}^{*}$ is bijective on
     $H^{1}(\mathcal{O}_{X})$ and $H^{2}(\mathcal{O}_{X})$ if and
     only if $F_{E}^{*}|H^{1}(\mathcal{O}_{E})$ and
     $F_{C}^{*}|H^{1}(\mathcal{O}_{C})$ are bijective. When $R =
     \Z/2$,
     $$H^{1}(\mathcal{O}_{X}) =
     H^{1}(\mathcal{O}_{C}) \textrm{
     and } H^{2}(\mathcal{O}_{X}) = H^{1} (\mathcal{O}_{E}) \otimes
     H^{1}(L_{1}).$$
     Under the assumption, $X$ is ordinary if and only if
     $F_{E}^{*}| H^{1}(\mathcal{O}_{E})$, $F_{C}^{*}|
     H^{1}(\mathcal{O}_{C})$ and $F_{D'}^{*}|H^{1}(L_{1})$ are
     bijective, i.e. $E$ and $D'$ are ordinary. When $R= \Z/3$,
     $$H^{1}(\mathcal{O}_{X}) = H^{1}(\mathcal{O}_{X'}) =
     H^{1}(\mathcal{O}_{C}), \ H^{2}(\mathcal{O}_{X}) = H^{1} (\mathcal{O}_{E}) \otimes
     H^{1}(L_{2})$$
     $$\textrm{and } H^{2}(\mathcal{O}_{X'}) = H^{1}(\mathcal{O}_{E}) \otimes
     H^{1}(L_{1}).$$
     Under the assumption, if $X$ and $X'$ are ordinary, $C$ and $E$ are ordinary and $F_{D'}^{*}$
     are bijective on $H^{1}(L_{1})$ and $H^{1}(L_{2})$. Hence
     $F_{D'}^{*}|( H^{1}(\mathcal{O}_{C}) \oplus H^{1}(L_{1})
     \oplus H^{1}(L_{2}))$ is bijective, so $D'$ is ordinary.
     The converse also holds.
     Note that when
     $R = \Z/3$, $j(E) = 0$ and $E$ is ordinary if and only if $p
     \equiv 1$ (mod 6). If so, $F_{D'}^{*}$ sends $L_{1}$ into $L_{1}$ and $L_{2}$
     into $L_{2}$ respectively. When $R =\Z/4$,
      $$H^{1}(\mathcal{O}_{X}) = H^{1}(\mathcal{O}_{X'}) =
     H^{1}(\mathcal{O}_{C}), \ H^{2}(\mathcal{O}_{X}) = H^{1} (\mathcal{O}_{E}) \otimes
     H^{1}(L_{3})$$
     $$\textrm{and } H^{2}(\mathcal{O}_{X'}) = H^{1}(\mathcal{O}_{E}) \otimes
     H^{1}(L_{1}).$$
     Under the assumption, if $X$ and $X'$ are ordinary, $C$ and $E$ are
     ordinary and $F_{D'}^{*}$ are bijective on $H^{1}(L_{1})$ and
     $H^{1}(L_{3})$. Again note that in this situation,
     $j(E)=1728$ and $p \equiv 1$ (mod 4). Let $\mathcal{O}_{D'} =
     \mathcal{O}_{D''} \oplus M$ as an $\mathcal{O}_{D''}$-module.
     $g''_{*} M = L_{1} \oplus L_{3}$. $F_{D'}^{*}
     |H^{1}(M)$ is bijective if and only if
     $$\gen D' - \gen D''=p \textrm{-rank of }D'-p\textrm{-rank of }D''.$$
     Hence (d) is valid. For (e), note that as an
     $\mathcal{O}_{D'''}$-module, $\mathcal{O}_{D'} =
     \mathcal{O}_{D'''} \oplus M$ where $M = L_{1} \oplus L_{3}
     \oplus L_{5}$ and $\mathcal{O}_{D''} = \mathcal{O}_{C} \oplus
     L_{3}$. The rest of proof is similar to that of (d).
     \end{proof}
     \begin{cor}
     Let $\pi : X \to C$ be an isotrivial elliptic fibration and $E$
     be the generic common fiber. If $X$ is ordinary and $E$ is
     supersingular, then $C= \P ^{1}$ and $X$ is rational.
     \end{cor}
     \begin{proof}
     By the proof of the above theorem, $H^{1}(R^{1} \pi _{*} \mathcal{O}_{X})$ should be 0. Since
     $\deg R^{1} \pi _{*} \mathcal{O}_{X} \leq 0$, $C = \P ^{1}$
     and $R^{1}\pi _{*} \mathcal{O}_{X} = \mathcal{O}_{C}$ or $\mathcal{O}(-1)$.
     If $R^{1} \pi _{*} \mathcal{O}_{X} = \mathcal{O}_{C}$, $R$ is
     trivial and $H^{1}(\mathcal{O}_{X}) =H^{1}(\mathcal{O}_{E})$,
     so $X$ is not ordinary. Hence $R^{1} \pi _{*}
     \mathcal{O}_{X} = \mathcal{O}(-1)$ and $X$ is rational.
     \end{proof}
     \begin{ex}
     There exists an rational isotrivial elliptic fibration
     whose generic fiber is supersingular. Let $E$ be a
     supersingular elliptic curve and $\rho _{E}$ be the action of $\Z /2$ on $E$ by $(-1)$-involution.
     $D = \P^{1}$ and $\rho _{D}$ is the action of $\Z/2$ given by
     $x \mapsto -x$ where $x$ is an affine coordinate of $\P ^{1}$.
     Let $X\to \P^{1}$ be the isotrivial
     elliptic fibraiton associated to the $\Z /2$ action $\rho_{E} \times \rho _{D}$ on $E
     \times D$. Then $E_{l}(X)=12$, so $X$ is rational and ordinary, but the
     generic fiber, $E$ is supersingular.
     \end{ex}
     \begin{ex}
     There exists an ordinary fibred surface $\pi : X \to C$ of
     fiber genus$\geq 2$ such that the generic fiber is not
     ordinary. There is a cyclic \'{e}tale covering of curves $D
     \to C$ of order $n$ such that $p \nmid n$, $C$ is ordinary
     and $D$ is not ordinary. (\cite{R},p.76) Let $G=T= \Z/n$. $G$ has the canonical action on $D$.
     Let $E$ be an ordinary
     elliptic curve. $P \in E(k)$ is an $n$-torsion point. $G$ acts on $E$ through the translation by $P$.
     Let $Y = (D \times E)/G$. $Y \to C$
     is an isotrivial elliptic fibration with $R=0$. Since $C$ and $E$ are ordinary,
     $Y$ is ordinary by Theorem 4.1.(a). But for the fibration $\pi : Y
      \to E/G$, every fiber of $\pi$ is isomorphic to
     $D$, so not ordinary.
     \end{ex}
     Now assume the generic fiber $E$ of an isotrivial elliptic
     fibration $\pi$ is ordinary. Let us recall the Frobenius
     diagram of $\pi$ (\ref{Frodia})
     $$\begin{diagram}
     X & \rTo ^{F_{X/C}} & X^{(p)} & \rTo ^{W} & X\\
     & \rdTo <{\pi} & \dTo <{\pi ^{(p)}}  & & \dTo <{\pi}\\
     & & C & \rTo ^{F _{C}} & C
     \end{diagram}
     $$
     and an exact sequence of $\mathcal{O}_{X^{(p)}}$-modules
     $$0 \to \mathcal{O}_{X^{(p)}} \to F_{X/C*} \mathcal{O}_{X} \to B \to 0.$$
     Since $\pi$ is generically ordinary, from (\ref{exac}) we obtain
     $$
     0 \to F_{C}^{*} \pi _{*}\mathcal{O}_{X} \overset
     {F_{X/C}^{*}}{\to} \pi _{*} \mathcal{O}_{X} \to R^{1} \pi^{(p)} _{*} B \to 0.
     $$
     Let $H=R^{1} \pi ^{(p)}_{*} B$. $H$ is a positive divisor concentrated on
     the image of singular fibers. Assume $d = -\deg R^{1} \pi_{*}
     \mathcal{O}_{X}$. The degree of $H$ is $d(p-1)$.
     \begin{defn}
     We say the divisor $H$, the Hasse divisor of $\pi$ and
     the multiplicity of $H$ at $x \in H$, the Hasse
     multiplicity of $\pi$ at $x$.
     \end{defn}
     \begin{prop}
     Let $\pi : X \to C$ be a minimal generically ordinary isotrivial elliptic
     fibration. For $x \in C$, $X_{x}$ is the fiber of $\pi$ over
     $x$. Then the Hasse multiplicity of $\pi$ at $x$ is
     $\frac{1}{12}(p-1) E_{l}(X_{x})$.
     \end{prop}
     \begin{proof}
     Assume $X_{x}$ is not a multiple fiber. The Hasse
     multiplicity at $x$ depends only on the type of $X_{x}$.
     Indeed, $X \otimes _{\mathcal{O}_{C}} \hat{\mathcal{O}}_{C,x}$
     is the relatively minimal desingularization of the quotient of $E \times _{k}
     \hat{\mathcal{O}}_{D,y}$ by $\Stab _{y}(G)$ where $y$ is a ramification point lying over $x$.
     It depends only on the type of the fiber $X_{x}$. Since the
     relative Frobenius morphism is compatible with the base
     change, the Hasse multiplicity is also determined by $X
     \otimes _{\mathcal{O}_{C}} \hat{\mathcal{O}}_{C,x}$. Now
     consider the case $D =\P ^{1}$, $G=R=\Z/2$ and $G$ action on $D$
     is given by $x
     \mapsto -x$, where $x$ is an affine coordinate of $\P ^{1}$.
     Let $X$ be the minimal isotrivial elliptic fibration associated to
     $(E \times D)/G$. Then $E_{l}(X)=12$ and the fibration has two fibers of type
     $I_{0}^{*}$.
     The degree of $H$ is $p-1$, so the Hasse
     multiplicity at the image of a fiber of type $I_{0}^{*}$ is
     $\frac{1}{2} (p-1)$ which is $\frac{1}{12}(p-1)$ times the Euler
     characteristic of the fiber. Assume $j(E)=1728$. $G=R =\Z/4$
     acts on $E$ via an isomorphism $\epsilon : \Z /4 \simeq \Aut E$. The
     minimal fibration $\pi : X \to \P ^{1}$ associated to $(E \times E)/G$ has 1 fiber of
     type $I_{0}^{*}$ and 2 fibers of type $III$. $E_{l}(X)=12$ and $\deg H
     =(p-1)$. Hence the Hasse multiplicity at the image of a fiber of
     type $III$ is $\frac{1}{4}(p-1)$. Now assume $E'=E$ and $G=\Z/4$ acts
     on $E'$ through the other isomorphism $-\epsilon : \Z/4 \to \Aut
     E$. The fibration $\pi ' :X' \to \P ^{1}$ associated to $(E \times E')/G$ has 1 fiber
     of type $I_{0}^{*}$ and two fibers of type $III^{*}$ and $E_{l}(X)=24$.
     Hence
     $\deg H =2(p-1)$ and the Hasse multiplicity at the image of
     the fiber of type $III ^{*}$ is $\frac{3}{4}(p-1)$. When
     $j(E)=0$, there are two actions of $\Z/3$ on $E$ and two
     actions of $\Z /6$ on $E$. It follows by a similar arguments to the above,
     that the Hasse invariants at the image of fiber of type $II$, $II
     ^{*}$, $IV$ and $IV ^{*}$ are $\frac{1}{6}(p-1)$,
     $\frac{5}{6}(p-1)$, $\frac{1}{3}(p-1)$ and $\frac{2}{3}(p-1)$
     respectively.
     \end{proof}
     \begin{rem}
     Let $h$ be the Hasse multiplicity of $\pi$ at $x \in C$. Let
     $\check{X} = X \otimes _{\mathcal{O}_{C}} \mathcal{O}_{C,x}$
     and $\check{\pi} : \check{X} \to \Spec \mathcal{O}_{C,x}$.
     Then $R\check{\pi}_{*} (B|\check{X})$ is quasi
     isomorphic to
     $$0 \to 0 \to \mathcal{O}_{C,x}/p^{h} \to 0,$$
     here the non-zero
     term occurs only at the 1st level.
     Therefore,
     $$RH (B|
     X_{x})=R \check{\pi}_{*} (B| \check{X})
     \overset{\mathbb{L}}{\otimes} (\mathcal{O}_{X,x}/m_{x})$$
     is quasi isomorphic to
     $$0 \to \mathcal{O}_{X,x}/m_{x}^{h}
     \overset{\times \pi _{x}}{\to} \mathcal{O}_{X,x}/m_{x}^{h} \to
     0,$$
     here $\pi _{x}$ is a uniformizer of $\mathcal{O}_{X,x}$.
     Hence
     $H^{0}(B|X_{x}) = H^{1}(B|X_{x}) =
     k \textrm{ if } h>0$ and both are 0 if h=0. Therefore the cohomologies of $B|
     X_{x}$ does not contain information about the Hasse
     multiplicity, but only the ``ordinarity'' of the fiber.
      Note that $B|X_{x}$ is the cokernel of the Frobenius morphism
     of the structure sheaf of $X_{x}$, so it depend only on the
     fiber $X_{x}$.
     The above proposition shows that for an isotrivial elliptic
     fibration with a ordinary generic fiber, the Hasse multiplicity
     of a singular fiber is just a multiple of the Euler
     characteristic of the fiber by a constant determined by the
     characteristic of the base field. However for a
     non-isotrivial elliptic fibration, this problem is subtle
     because of the existence of smooth supersingular fibers.
     \end{rem}
     \begin{rem}
     The Frobenius morphism on $H^{2}(\mathcal{O}_{X})$ is
     determined by $R^{1} \pi _{*}\mathcal{O}_{X}$ and the Hasse
     divisor. Recall from the Frobenius diagram (\ref{Frodia}), $F_{X} = W
     \circ F_{X/C}$ and $H^{2}(\mathcal{O}_{X}) = H^{1}(R^{1} \pi
     _{*} \mathcal{O}_{X})$. $W ^{*} : H^{2}(\mathcal{O}_{X}) \to
     H^{2} (\mathcal{O} _{X^{(p)}})$ is equal to the $H^{1}$ of
     the morphism
     $$F_{C}^{*} : R^{1}\pi _{*} \mathcal{O}_{X} \to
     F_{C*}\mathcal{O}_{C} \otimes R^{1}\pi _{*} \mathcal{O}_{X}$$
     which depends only on $R^{1} \pi _{*} \mathcal{O}_{X}$. And
     $F_{X/C}^{*} : H^{2}(\mathcal{O}_{X^{(p)}}) \to
     H^{2}(\mathcal{O}_{X})$ is equal to the $H^{1}$ of
     $$F_{X/C}^{*} :F_{C}
     ^{*} R^{1}\pi _{*}\mathcal{O}_{X} \overset{D}{\to}
     R^{1}\pi_{*}\mathcal{O}_{X}.$$ When $C =\P ^{1}$, $R
     =\Z/2$ and $D'$ is a hyperelliptic curve, these morphisms can
     be expressed precisely. Let $[x,y]$ be a projective coordinate
     of $\P ^{1}$.
     Assume $R^{1} \pi _{*} \mathcal{O}_{X} = \mathcal{O}(-d)$.
     $$\frac{1}{x^{d-1}y}, \frac{1}{x^{d-2}y^{2}} , \cdots ,
     \frac{1}{xy^{d-1}}$$
     forms a basis of $H^{1}(\mathcal{O}(-d))$. The image of $W^{*}
     : H^{1}(\mathcal{O}(-d)) \to H^{1}(\mathcal{O}(-pd))$ is
     generated by
     $$\frac{1}{x^{p(d-1)}y^{p}}, \frac{1}{x^{p(d-2)}y^{2p}} , \cdots ,
     \frac{1}{x^{p}y^{p(d-1)}}.$$
     We may regard the Hasse divisor $D$ as a homogeneous
     polynomial in $x,y$ of degree $(p-1)d$, say $k[x,y]=\sum _{i=0}^{(p-1)d} a_{i}x^{i}y^{(p-1)d-i}$. The image of
     $\frac{1}{x^{j}y^{k}}$ by\\
     $F_{X/C}^{*} :
     H^{1}(\mathcal{O}(-pd)) \to H^{1}(\mathcal{O}(-d))$ is
     $$  a_{j-1}\frac{1}{xy^{d-1}} + \cdots +
     a_{j-d+1}\frac{1}{x^{d-1}y},$$
     the sum of the $\frac{1}{x^{d-1}y}, \frac{1}{x^{d-2}y^{2}}, \cdots ,
     \frac{1}{xy^{d-1}}$ terms of $\frac{k[x,y]}{x^{j}y{k}} = \sum
     _{i=0}^{(p-1)d} a_{i} \frac{x^{i}y^{(p-1)d-i}}{x^{j}y^{k}}$.
     If we arrange the coefficients of $\frac{1}{x^{d-1}y},\cdots ,
     \frac{1}{xy^{d-1}}$ of $F_{X/C}^{*}(\frac{1}{x^{jp}y^{kp}})$ properly, it makes the Hasse
     matrix of the hyperelliptic curve $D'$. Hence $X$ is ordinary
     if and only if $D'$ is ordinary. This recovers Theorem 4.1.(b)
     for this case.
     \end{rem}
     \begin{thm} Let $\pi : X \to C$ be an isotrivial elliptic
     fibration associated to an action of a finite group $G = T \rtimes R$ on $E
     \times D$ defined over a number field $F$ and $D' =
     D/T$.\\[0.15cm]
     (a) If $R$ is trivial, the ordinary reduction theorem holds
     for $X$ if and only if it holds for $C$.\\[0.15cm]
     (b) If $R$ is non-trivial, the ordinary reduction theorem holds for
     $X$ if it holds for $D'$.
     \end{thm}
     \begin{proof}
     Since the ordinary reduction theorem holds for $E$, there are
     a number field $F' \supset F$ and $\mathcal{M}_{F'}$, a set of places
     of $F'$ of density 1, such that for each $\upsilon \in \mathcal{M}_{F'}$,
     the reductions of $E \otimes F'$ is
     ordinary. If $C$ satisfies the ordinary reduction theorem,
     there are also a number field $F'' \supset F$ and
     $\mathcal{M}_{F''}$,
     a set of places of $F''$ of density 1, such that for each $\upsilon \in \mathcal{M}_{F''}$,
     the reduction of $C \otimes F''$
     is ordinary. Hence for a number field $F''' \supset F' \cup
     F''$, the set of places of $F'''$, $\mathcal{M}_{F''} \cap \mathcal{M}_{F'''}$
     has density 1 and at a place in $\mathcal{M}_{F''} \cap \mathcal{M}_{F'''}$,
     the reductions of both of $C$ and $E$ are ordinary. By
     theorem 4.1, (a) is valid. The same argument can be applied
     to (b)
     \end{proof}
     \begin{cor}
     Let $\pi : X \to C$ be an isotrivial elliptic fibration defined
     over a number field $F$. Assume each non-multiple singular
     fiber of $\pi$ is of type $I_{0}^{*}$.\\[0.15cm]
     (a) If $\gen C=0$ and $\chi(\mathcal{O}_{X})\leq 3$, then the
     ordinary reduction theorem holds for $X$.\\[0.15cm]
     (b) If $\gen C =1$ and $\chi(\mathcal{O}_{X}) \leq 1$, then
     the ordinary reduction theorem holds for $X$.\\[0.15cm]
     (c) If $\gen C =2$ and $R^{1} \pi _{*} \mathcal{O}_{X} =
     \mathcal{O}_{C}$, then the ordinary reduction theorem holds
     for $X$.
     \end{cor}
     \begin{proof}
     Note that if $j(E) \neq 0,1728$, the assumption is valid.
     Since a proper smooth curve of genus $\leq 2$ satisfies the
     ordinary reduction theorem, by the above theorem, the claim
     follows.
     \end{proof}


\begin{thebibliography}{99}
     \bibitem{BPV} Barth,W., Peters,C. and van de Ven,A., Compact complex surfaces,
      {\em Ergeb. Math. Grenzgeb. (3)}
   {\bf 4}, (1984).
     \bibitem{Be} Beauville,A., Surfaces alg\'{e}briques complexes,
      {\em Ast\'{e}risque}
   {\bf 54}, (1978).
     \bibitem{BZ} Bogomolov,F. and Zarhin,Y., Ordinary reduction of $K3$
     surfaces, {\em Cent. Eur. J. Math.}
   {\bf 7}, (2009), no.3, 73--212.
     \bibitem{Ek} Ekedahl,T., On the multiplicative properties of the de Rham-Witt complex.
     II, {\em Ark. Mat.}
   {\bf 23}, (1985), no.1, 53--102.
     \bibitem{GK} van der Geer,G. and Katsura,T., On the height of Calabi-Yau varieties in positive characteristic,
     {\em Doc.Math.}
   {\bf 8}, (2003), no.3, 97--113.
     \bibitem{I} Illusie,L. Complexe de de Rham-Witt et
     cohomologie cristalline, {\em Ann.ENS 4serie} {\bf 12}, (1979), 501--661.
     \bibitem{IR} Illusie,L. and Raynaud,M., Les suites spectrales
     associ\'{e}es au complexe de de Rham-Witt, {\em Pub. IHES}
     {\bf 57}, (1983), 73--212
     \bibitem{JR} Joshi,K. and Rajan,C., Frobenius splitting and ordinarity,
     {\em Int. Math. Res. Not.}, (2003), no.2, 109--121.
     \bibitem{No} Noot,R., Abelian varieties-Galois representation and properties of ordinary reduction,
     {\em Compositio Math.}
     {\bf 97}, (1995), no.1-2, 367--414.
      \bibitem{Ogu} Oguiso,K., An elementary proof of the topological Euler characteristic formula for an elliptic surface,
     {\em Comment. Math. Univ. St. Paul.}, {\bf 39}, (1990), 81--86.
     \bibitem{Og} Ogus,A., Hodge cycles and Crystalline Cohomology, {\em Lecture Notes in Mathematics}
     {\bf 900}, (1982), 367--414.
     \bibitem{Pi} Pink,R., $l$-adic algebraic monodromy groups, cocharacters, and the Mumford-Tate conjecture,
     {\em J. Reine Angew. Math.}
     {\bf 495}, (1998), 187--237.
     \bibitem{R} Raynaud,M., Rev\^{e}tements des courbes en caract\'{e}ristique $p>0$ et ordinarit\'{e},
     {\em Compositio Math.}
     {\bf 123}, (2000), 73--88.
     \bibitem{Ser} Serrano,F., Isotrivial Fibred surface,
     {\em Ann. Mat. Pura Appl.(4)}
     {\bf 171}, (1996), 63--81.
     \bibitem{Se} Serre,J.P., Groupes de Lie l-Adiques
     Attache\'{e}s aux Courbes Elliptiques, {\em Colloque de
     Clermont-Ferrand, IHES}, (1964).
     \bibitem{Si} Silverman,J., The Arithmetic of Elliptic
     Curves, {\em Grad. Texts in Math.}, (1986).
     \end{thebibliography}
     \end{document}